\documentclass{article}

\PassOptionsToPackage{numbers}{natbib}
\usepackage[preprint]{neurips_2026}

\usepackage[utf8]{inputenc}
\usepackage[T1]{fontenc}
\usepackage{lmodern}
\usepackage{microtype}
\usepackage{hyperref}
\usepackage{url}
\usepackage{booktabs}
\usepackage{amsfonts}
\usepackage{nicefrac}
\usepackage{xcolor}
\usepackage{amsmath,amssymb,amsthm,mathtools,bm}
\usepackage{enumitem}
\usepackage{graphicx}
\usepackage{array}
\usepackage{adjustbox}
\hypersetup{hidelinks}

\newtheorem{theorem}{Theorem}[section]
\newtheorem{lemma}[theorem]{Lemma}
\newtheorem{proposition}[theorem]{Proposition}
\newtheorem{corollary}[theorem]{Corollary}
\theoremstyle{definition}
\newtheorem{definition}[theorem]{Definition}

\newcommand{\RR}{\mathbb R}
\newcommand{\EE}{\mathbb E}
\newcommand{\one}{\mathbf 1}

\title{Order-Sensitive Sequential Interventions on Ideal Lattices}

\author{%
  Dmitry Pasechnyuk-Vilensky\\
  MBZUAI, UAE \\
  \texttt{dmivilensky1@gmail.com} \\
}

\begin{document}

\maketitle

\begin{abstract}
We study sequential interventions under prerequisite constraints. In this setting, admissible intervention sequences are paths in the ideal lattice of a finite prerequisite poset rather than unconstrained action strings. We give an exact local-to-global theory of order sensitivity on this state space. First, we prove that any two admissible paths with the same endpoints differ by a finite sequence of elementary diamond swaps. Second, for edge-additive path valuations, we show that path-independence is equivalent to vanishing diamond curvature, yielding an endpoint potential with a canonical M\"obius parameterization on the ideal lattice. Third, we prove that a local diamond field is induced by an edge-based path model if and only if it satisfies cube consistency, with uniqueness after fixing a reference-tree gauge. Under reduced-state longitudinal assumptions, supported reference paths identify reference-path scores, whereas local order effects require two-sided support of both orders on each diamond. These results yield exact planning consequences, including an order-insensitivity bound and dynamic programming on the truncated ideal lattice.
\end{abstract}

\section{Introduction}

Many sequential intervention problems are constrained by prerequisite structure. In tutoring, remediation steps must respect skill dependencies; in treatment design, later actions may be admissible only after earlier conditions are satisfied; in workflow control, tasks may be executed only after their dependencies are met. In all such cases, the admissible action space is governed by a partial order rather than by unconstrained sequence composition.

We model this structure by a finite poset
\[
P=(S,\preceq).
\]
The natural state space is the ideal lattice \(J(P)\): a state records a downward-closed set of currently satisfied prerequisites, and an admissible action adds one new element whose predecessors are already present. The representation of finite distributive lattices by ideal families is classical \cite{birkhoff1937,davey2002,stanley2011ec1}. Thus admissible intervention sequences are naturally represented as paths in a finite distributive lattice.

The minimal nontrivial instance already contains the main phenomenon. Let \(u\) and \(v\) be incomparable admissible additions above a base ideal \(I\). Then the two admissible paths
\[
I\to I\cup\{u\}\to I\cup\{u,v\},
\qquad
I\to I\cup\{v\}\to I\cup\{u,v\},
\]
share the same endpoint but may induce different downstream value. This rank-two Boolean interval is the elementary {diamond}. The paper proves that all same-endpoint order effects reduce to compositions of such local diamonds.

This reduction is exact. Admissible paths in a fixed interval \([I,J]\) are in bijection with linear extensions of the induced poset on \(J\setminus I\), and the graph of linear extensions is connected by adjacent swaps of neighboring incomparable elements \cite{trotter1992,pruesse1994,stanley2011ec1}. In the ideal lattice, each such swap is exactly a diamond move. This is the basic structural fact behind the paper: local diamond geometry completely controls global order sensitivity between paths with the same endpoints.

This combinatorial fact yields an exact local-to-global theory of order-sensitive sequential interventions on ideal lattices. For a fixed structural context \(x\), a path value is encoded by an edge field \(g_x\) on admissible edges of \(J(P)\). The associated {diamond curvature} \(\kappa_x\) measures the value difference between the two sides of a diamond. Vanishing curvature is exactly the integrable regime: path values depend only on endpoints and are represented by an endpoint potential \(\Phi_x\). This potential admits a canonical M\"obius parameterization in the incidence algebra of \(J(P)\) by standard M\"obius inversion on finite posets \cite{rota1964,stanley2011ec1}.

The second structural pillar is a realizability theorem. A local diamond field \(\kappa\) is globally realizable by an edge-based path model if and only if it satisfies cube consistency. The proof is constructive: after fixing a gauge on a reference tree, the realizing edge field is recovered recursively by peeling off \(<_{\tau}\)-maximal elements. Cube consistency is exactly the compatibility condition that makes this reconstruction well-defined.

The same structural decomposition has a sharp causal consequence. Under standard reduced-state longitudinal assumptions from dynamic treatment regimes and the \(g\)-formula literature \cite{robins1986,robins1987,murphy2003,murphy2005,hernanrobins2020}, supported reference paths identify reference-path scores, whereas local order effects require two-sided support of both local orders on each diamond. This yields a {support separation} theorem: reference-path and order-sensitive components are learnable on different support sets.

\paragraph{Related work and positioning.}
Three lines of work are closest to ours. Knowledge-space and learning-space theories study feasible state families induced by prerequisite relations, but not order-sensitive intervention values on paths through such states \cite{doignon1999,doignon2011}. Dynamic treatment regimes provide the causal and sequential decision-theoretic background, but do not exploit the geometry induced by prerequisite constraints \cite{robins1986,robins1987,murphy2003,murphy2005,qian2011}. Finite distributive lattices, linear extensions, and M\"obius inversion provide the combinatorial substrate \cite{birkhoff1937,davey2002,trotter1992,pruesse1994,rota1964,stanley2011ec1}. Our contribution is an exact theory at this intersection: diamonds generate all same-endpoint order variation, cube consistency characterizes globally realizable local order fields, and these structural quantities have distinct identification requirements.

\paragraph{Contributions.}
The paper makes four contributions. It introduces ideal-lattice path valuation as the natural formal object for prerequisite-constrained sequential interventions. It proves a diamond representation theorem and an exact curvature decomposition for same-endpoint path differences. It proves a realizability theorem for local diamond fields, with cube consistency as the exact global compatibility condition and uniqueness after reference-tree gauge fixing. Finally, it derives a causal support separation theorem and exact planning consequences, including an order-insensitivity bound and dynamic programming on the truncated ideal lattice.

\paragraph{Limitations.}
The theory is developed for finite posets, finite-horizon admissible paths, and edge-additive structural means. The causal results are identification results. Richer state dependence, non-additive path functionals, and continuous action spaces are outside the present scope.

The paper is organized as follows. Section~\ref{sec:state-space} defines ideal-lattice state spaces and admissible paths. Section~\ref{sec:diamonds} proves that same-endpoint path differences are generated by diamonds. Section~\ref{sec:curvature} develops curvature, endpoint potentials, and exact decomposition. Section~\ref{sec:integrability} proves the realizability theorem for local diamond fields. Section~\ref{sec:causal} establishes the support separation theorem under longitudinal causal assumptions. Section~\ref{sec:planning} gives planning consequences. All proofs are deferred to the appendices.

\section{Ideal-Lattice State Space and Admissible Paths}
\label{sec:state-space}

Fix a finite poset
\[
P=(S,\preceq).
\]
Fix once and for all a total order
\[
\tau=(s_1,\dots,s_{|S|})
\]
extending \(\preceq\); write \(u<_{\tau}v\) when \(u\) precedes \(v\) in \(\tau\).

\begin{definition}[Ideals]
An {ideal} of \(P\) is a set \(I\subseteq S\) such that
\[
s\in I,\ u\preceq s \Longrightarrow u\in I.
\]
The set of all ideals is denoted by \(J(P)\), ordered by inclusion.
\end{definition}

\begin{definition}[Admissible additions]
For \(I\in J(P)\), define
\[
\mathcal A(I):=\{a\in S\setminus I:\ \mathrm{Pred}(a)\subseteq I\},
\qquad
\mathrm{Pred}(a):=\{u\in S:\ u\prec a\}.
\]
\end{definition}

\begin{definition}[Admissible paths]
Let \(I,J\in J(P)\) with \(I\subseteq J\). An {admissible path} from \(I\) to \(J\) is a sequence
\[
\gamma=(I_0,\dots,I_L)
\]
such that
\[
I_0=I,\qquad I_L=J,\qquad I_{\ell+1}=I_\ell\cup\{a_\ell\},\quad a_\ell\in \mathcal A(I_\ell).
\]
The set of all such paths is denoted by \(\Gamma(I,J)\). For \(H\in\mathbb N\), let \(\Gamma_H(I,J)\subseteq \Gamma(I,J)\) be the subset of paths of length at most \(H\).
\end{definition}

\begin{definition}[Reference paths]
For \(I\subseteq J\) in \(J(P)\), let
\[
J\setminus I=\{a_1,\dots,a_m\},
\qquad
a_1<_{\tau}\cdots<_{\tau} a_m.
\]
The {reference path} from \(I\) to \(J\) is
\[
\rho_{I,J}:\ 
I
\to I\cup\{a_1\}
\to I\cup\{a_1,a_2\}
\to \cdots \to
I\cup\{a_1,\dots,a_m\}=J.
\]
\end{definition}

\begin{definition}[Edge fields and path valuations]
Fix \(x\) in a structural context space \(\mathcal X\). An {edge field} is a function
\[
g_x:\{(I,a): I\in J(P),\ a\in \mathcal A(I)\}\to \RR.
\]
The induced valuation of an admissible path
\[
\gamma=(I_0,\dots,I_L),
\qquad
I_{\ell+1}=I_\ell\cup\{a_\ell\},
\]
is
\begin{equation}
\textstyle V_x(\gamma):=\sum_{\ell=0}^{L-1} g_x(I_\ell,a_\ell).
\label{eq:path-valuation}
\end{equation}
\end{definition}

A classical fact is that \(J(P)\) is a finite distributive lattice \cite{birkhoff1937,davey2002,stanley2011ec1}. The next section identifies the elementary local moves in this lattice.

\section{Diamonds and Rewriting of Same-Endpoint Paths}
\label{sec:diamonds}

\begin{definition}[Diamonds]
Let \(I\in J(P)\) and let \(u,v\in\mathcal A(I)\) satisfy
\[
u<_{\tau}v,\qquad u\parallel v.
\]
The corresponding {diamond} is denoted by
\[
d=(I;u,v)
\]
and has boundary paths
\[
I\to I\cup\{u\}\to I\cup\{u,v\},
\qquad
I\to I\cup\{v\}\to I\cup\{u,v\}.
\]
\end{definition}

The proof of the next theorem uses the linear-extension representation of admissible paths and adjacent-swap connectivity \cite{trotter1992,pruesse1994,stanley2011ec1}.

\begin{theorem}[Diamond representation]
\label{thm:diamond-representation-main}
Let \(I\subseteq J\) in \(J(P)\), and let \(\gamma,\gamma'\in \Gamma(I,J)\). Then \(\gamma'\) is obtained from \(\gamma\) by a finite sequence of diamond swaps. Equivalently, there exist diamonds
\[
d_1,\dots,d_M
\]
and signs \(\varepsilon_m\in\{\pm1\}\) such that
\begin{equation}
\textstyle \gamma-\gamma'=\sum_{m=1}^M \varepsilon_m\,\partial d_m
\label{eq:diamond-representation}
\end{equation}
as formal \(1\)-chains.
\end{theorem}

\begin{center}
\fbox{
\begin{minipage}{0.93\linewidth}
\textbf{Main Result 1.}
Any two admissible paths with the same endpoints differ by a finite sequence of diamond swaps.
\end{minipage}
}
\end{center}

Theorem~\ref{thm:diamond-representation-main} identifies diamonds as the elementary generators of same-endpoint order variation.

\section{Curvature, Endpoint Potentials, and Exact Decomposition}
\label{sec:curvature}

\begin{definition}[Diamond curvature]
For a diamond \(d=(I;u,v)\) with \(u<_{\tau}v\), define
\begin{equation}
\kappa_x(d)
:=
g_x(I,v)+g_x(I\cup\{v\},u)-g_x(I,u)-g_x(I\cup\{u\},v).
\label{eq:diamond-curvature}
\end{equation}
Equivalently,
\[
\kappa_x(d)
=
V_x\!\left(I\to I\cup\{v\}\to I\cup\{u,v\}\right)
-
V_x\!\left(I\to I\cup\{u\}\to I\cup\{u,v\}\right).
\]
\end{definition}

\begin{theorem}[Path-independence criterion]
\label{thm:path-independence-main}
Fix \(x\in\mathcal X\). The following are equivalent:
\begin{enumerate}
    \item for every \(I\subseteq J\) in \(J(P)\) and all \(\gamma,\gamma'\in \Gamma(I,J)\),
    \begin{equation}
    V_x(\gamma)=V_x(\gamma');
    \label{eq:path-independence-condition}
    \end{equation}
    \item for every diamond \(d\),
    \begin{equation}
    \kappa_x(d)=0;
    \label{eq:zero-curvature-condition}
    \end{equation}
    \item there exists a function \(\Phi_x:J(P)\to\RR\) such that for every admissible edge \(I\to I\cup\{a\}\),
    \begin{equation}
    g_x(I,a)=\Phi_x(I\cup\{a\})-\Phi_x(I).
    \label{eq:endpoint-potential-gradient}
    \end{equation}
\end{enumerate}
In this case,
\begin{equation}
V_x(\gamma)=\Phi_x(J)-\Phi_x(I)
\qquad
\forall \gamma\in \Gamma(I,J).
\label{eq:path-independence-telescoping}
\end{equation}
\end{theorem}

Thus vanishing curvature is exactly the integrable regime.

\begin{definition}[Reference-path scores]
For \(I\subseteq J\), define
\begin{equation}
\Phi_x^\rho(I,J):=V_x(\rho_{I,J}).
\label{eq:reference-path-score}
\end{equation}
\end{definition}

\begin{theorem}[Exact reference-path decomposition]
\label{thm:reference-decomposition-main}
Fix \(x\in\mathcal X\). Let \(I\subseteq J\) in \(J(P)\), and let \(\gamma\in \Gamma(I,J)\). Then
\begin{equation}
\textstyle V_x(\gamma)
=
\Phi_x^\rho(I,J)
+
\sum_{m=1}^M \varepsilon_m\,\kappa_x(d_m),
\label{eq:reference-path-decomposition}
\end{equation}
where \((d_m,\varepsilon_m)_{m=1}^M\) is any diamond-swap sequence taking \(\rho_{I,J}\) to \(\gamma\).
\end{theorem}

Theorem~\ref{thm:reference-decomposition-main} is the value-theoretic form of Theorem~\ref{thm:diamond-representation-main}: once same-endpoint path differences are generated by diamonds, path values decompose exactly into a reference-path score plus signed local curvature corrections.

The next statement is standard M\"obius inversion on the finite distributive lattice \(J(P)\) \cite{rota1964,stanley2011ec1}.

\begin{corollary}[M\"obius parameterization of the integrable regime]
\label{cor:mobius-main}
Under the equivalent conditions of Theorem~\ref{thm:path-independence-main}, there exists a unique function \(\theta_x:J(P)\to\RR\) such that
\begin{equation}
\textstyle \Phi_x(I)=\sum_{K\subseteq I}\theta_x(K),
\label{eq:mobius-forward}
\end{equation}
and
\begin{equation}
\textstyle \theta_x(I)=\sum_{K\subseteq I}\mu(K,I)\,\Phi_x(K),
\label{eq:mobius-inverse}
\end{equation}
where \(\mu\) is the M\"obius function of \(J(P)\). In particular,
\begin{equation}
\textstyle g_x(I,a)=\sum_{\substack{K\subseteq I\cup\{a\}\\ a\in K}} \theta_x(K).
\label{eq:mobius-edge}
\end{equation}
\end{corollary}

\section{Cube Consistency and Global Realizability}
\label{sec:integrability}

\begin{definition}[Three-cubes]
Let \(I\in J(P)\) and let
$
u<_{\tau}v<_{\tau}w
$
be pairwise incomparable elements of \(\mathcal A(I)\). The corresponding {three-cube} is the Boolean interval
$
[I,\ I\cup\{u,v,w\}],
$
whose six \(2\)-faces are diamonds.
\end{definition}

\begin{definition}[Cube consistency]
A diamond field
$
\kappa:\{\text{diamonds of }J(P)\}\to \RR
$
is {cube-consistent} if for every three-cube \((I;u,v,w)\) with \(u<_{\tau}v<_{\tau}w\),
\begin{equation}
\kappa(I;u,v)
-\kappa(I\cup\{w\};u,v)
+\kappa(I;u,w)
-\kappa(I\cup\{v\};u,w)
+\kappa(I;v,w)
-\kappa(I\cup\{u\};v,w)
=0.
\label{eq:cube-consistency}
\end{equation}
\end{definition}

\begin{proposition}[Discrete Bianchi identity]
\label{prop:bianchi-main}
If \(\kappa_x\) is the curvature field of an edge field \(g_x\), then \(\kappa_x\) is cube-consistent.
\end{proposition}

\begin{definition}[Reference-tree scores]
For every nonempty ideal \(K\), let \(m(K)\) denote the \(<_{\tau}\)-maximum element of \(K\), and define
$
K^-:=K\setminus\{m(K)\}
$.
A {reference-tree score system} is a function
$
\alpha:J(P)\setminus\{\varnothing\}\to \RR,
\alpha(K):=g(K^-,m(K)).
$
\end{definition}

\begin{theorem}[Integrability with reference-tree gauge]
\label{thm:integrability-main}
Let \(\kappa\) be a cube-consistent diamond field, and let \(\alpha\) be a reference-tree score system. Then there exists a unique edge field \(g\) such that:
\begin{enumerate}
    \item \(g(K^-,m(K))=\alpha(K)\) for every nonempty ideal \(K\);
    \item the curvature field of \(g\) equals \(\kappa\).
\end{enumerate}
\end{theorem}

\begin{center}
\fbox{
\begin{minipage}{0.93\linewidth}
\textbf{Main Result 2.}
A local diamond field is globally realizable if and only if it satisfies the cube constraints; once the reference-tree edge values are fixed, the realizing edge field is unique.
\end{minipage}
}
\end{center}

\section{Causal Identification and Support Separation}
\label{sec:causal}

The causal setup follows the standard longitudinal intervention and dynamic treatment regime framework \cite{robins1986,robins1987,murphy2003,murphy2005,hernanrobins2020,qian2011}.

Fix a decision time \(t\). Let $Z_t=(I_t,X_t)$
be the reduced state. Let \(R^\gamma\) denote the potential reward under admissible path \(\gamma\).

\begin{definition}[Initial-state conditional value]
For \(z=(I,x)\), define
\begin{equation}
Q_z(\gamma):=\EE[R^\gamma\mid Z_t=z]
\label{eq:initial-state-conditional-value}
\end{equation}
for every admissible path \(\gamma\) starting from \(I\).
\end{definition}

\paragraph{Assumptions.}
We impose the following conditions.

\begin{enumerate}
    \item[\textbf{C1.}] \textbf{Consistency and no interference.}
    If the observed continuation equals \(\gamma\), then the observed reward equals \(R^\gamma\), and one unit's intervention does not affect another's reward.

    \item[\textbf{C2.}] \textbf{Reduced-state sequential ignorability.}
    At each stage \(b\),
    \begin{equation}
    A_b \perp\!\!\!\perp \{R^{\gamma'}:\gamma' \text{ admissible from } Z_b\}\mid Z_b.
    \label{eq:reduced-state-ignorability}
    \end{equation}

    \item[\textbf{C3.}] \textbf{Positivity.}
    Every admissible action or local order whose effect is to be identified occurs with positive conditional probability on the support of the corresponding intervention-induced state law.

    \item[\textbf{C4.}] \textbf{Statewise additive structural mean.}
    For every initial state \(z\), there exists an edge field \(g_z\) on the truncated ideal lattice reachable from the initial ideal such that
    \begin{equation}
    \textstyle Q_z(\gamma)=\sum_{e\in\gamma} g_z(e)
    \label{eq:statewise-additive-mean}
    \end{equation}
    for every admissible path \(\gamma\) from that initial ideal.
\end{enumerate}

\begin{definition}[Intervention-induced state law]
Let
$
\gamma=(a_0,\dots,a_{L-1})
$
be an admissible path from \(z_0\). Define recursively the path-induced state laws
\begin{equation}
\textstyle p_0^\gamma(\cdot\mid z_0):=\delta_{z_0},
\qquad
p_{b+1}^\gamma(z_{b+1}\mid z_0)
:=
\int p(z_{b+1}\mid z_b,a_b)\,p_b^\gamma(z_b\mid z_0)\,dz_b.
\label{eq:path-induced-state-law}
\end{equation}
\end{definition}

\begin{definition}[Supported paths]
An admissible path
$
\gamma=(a_0,\dots,a_{L-1})
$
starting from \(z_0\) is {supported at \(z_0\)} if, for every stage \(b=0,\dots,L-1\),
\begin{equation}
\Pr(A_b=a_b\mid Z_b=z_b)>0
\quad\text{for }p_b^\gamma(\cdot\mid z_0)\text{-a.e. }z_b.
\label{eq:path-support}
\end{equation}
\end{definition}

\begin{definition}[Two-sided supported diamonds]
A diamond \(d=(I;u,v)\), \(u<_{\tau}v\), is {two-sided supported at \(z=(I,x)\)} if both two-step paths
\[
I\to I\cup\{u\}\to I\cup\{u,v\},
\qquad
I\to I\cup\{v\}\to I\cup\{u,v\},
\]
are supported at \(z\).
\end{definition}

\begin{definition}[Local order effect at state \(z\)]
For a diamond \(d=(I;u,v)\), \(u<_{\tau}v\), define
\begin{equation}
\kappa_z(d)
:=
Q_z\!\left(I\to I\cup\{v\}\to I\cup\{u,v\}\right)
-
Q_z\!\left(I\to I\cup\{u\}\to I\cup\{u,v\}\right).
\label{eq:local-order-effect}
\end{equation}
\end{definition}

\begin{proposition}[Non-identifiability of unsupported local order effects]
\label{prop:nonid-main}
Fix \(I\in J(P)\), \(u<_{\tau}v\), and a measurable set \(B\subseteq \mathcal X\) with positive probability. Suppose that, for every \(x\in B\),
\[
\Pr\!\left(
I\to I\cup\{v\}\to I\cup\{u,v\}
\ \middle|\ 
Z_t=(I,x)
\right)=0.
\]
Then the local order effect \(\kappa_{(I,x)}(I;u,v)\) is not point-identified on \(B\) from the observational law over the class of models satisfying \textbf{C1}--\textbf{C3}.
\end{proposition}

\begin{theorem}[Identification of supported path values]
\label{thm:gformula-main}
Under \textbf{C1}--\textbf{C3}, let
$
\gamma=(a_0,\dots,a_{L-1})
$
be supported at \(z_0=(I_0,x_0)\). Then
\begin{align}
\textstyle Q_{z_0}(\gamma)
=
&\textstyle \int
\EE\!\left[
R
\mid
Z_L=z_L,\ A_0=a_0,\dots,A_{L-1}=a_{L-1},\ Z_0=z_0
\right] \times\\
&\textstyle \times
\prod_{b=0}^{L-1} p(z_{b+1}\mid z_b,a_b)\,dz_{1:L}.
\label{eq:gformula-supported-path}
\end{align}
\end{theorem}

\begin{theorem}[Support separation]
\label{thm:support-separation-main}
Assume \textbf{C1}--\textbf{C3}. Let \(z=(I,x)\).
\begin{enumerate}
    \item If \(\rho_{I,J}\) is supported at \(z\), then the reference-path score
    \begin{equation}
    \textstyle \Phi_z^\rho(I,J):=Q_z(\rho_{I,J})
    \label{eq:identified-reference-score}
    \end{equation}
    is identified.

    \item If \(d=(I;u,v)\) is two-sided supported at \(z\), then \(\kappa_z(d)\) is identified.
\end{enumerate}
If, in addition, \textbf{C4} holds, then the following stronger statement is valid:
\begin{enumerate}[resume]
    \item If \(\rho_{I,J}\) is supported at \(z\) and every diamond appearing in a rewrite from \(\rho_{I,J}\) to \(\gamma\in\Gamma(I,J)\) is two-sided supported at \(z\), then
    \begin{equation}
    \textstyle Q_z(\gamma)
    =
    \Phi_z^\rho(I,J)
    +
    \sum_{m=1}^M \varepsilon_m\,\kappa_z(d_m)
    \label{eq:identified-full-decomposition}
    \end{equation}
    is identified.
\end{enumerate}
\end{theorem}

\begin{center}
\fbox{
\begin{minipage}{0.93\linewidth}
\textbf{Main Result 3.}
Reference-path scores and local order effects have different support requirements: supported reference paths identify reference-path terms, whereas two-sided diamond support is necessary and sufficient for local order effects.
\end{minipage}
}
\end{center}

\begin{corollary}[Identification of the optimal policy on a supported class]
\label{cor:policy-main}
Fix \(H<\infty\). For each initial ideal \(I\), the set
$
\textstyle \Gamma_H(I):=\bigcup_{J\supseteq I} \Gamma_H(I,J)
$
is finite. Assume \textbf{C1}--\textbf{C4}. If, for every candidate \(\gamma\in \Gamma_H(I)\), the corresponding reference path \(\rho_{I,J_\gamma}\) is supported at the current state \(z=(I,x)\), and every diamond appearing in a rewrite from \(\rho_{I,J_\gamma}\) to \(\gamma\) is two-sided supported at \(z\), then
\begin{equation}
\textstyle \pi^\star(z)\in \arg\max_{\gamma\in \Gamma_H(I)} Q_z(\gamma)
\label{eq:identified-optimal-policy}
\end{equation}
is point-identified, after deterministic tie-breaking.
\end{corollary}

\section{Planning Consequences}
\label{sec:planning}

\begin{theorem}[Order-insensitivity bound]
\label{thm:order-bound-main}
Assume \textbf{C4}. Fix \(z=(I,x)\) and \(H\in\mathbb N\). Suppose that
$
|\kappa_z(d)|\le \varepsilon
$
for every diamond reachable from the initial ideal within depth \(H\). Then for any \(J\supseteq I\), any \(L\le H\), and any \(\gamma,\gamma'\in \Gamma_L(I,J)\),
\begin{equation}
\textstyle |Q_z(\gamma)-Q_z(\gamma')|
\le
N_{\mathrm{swap}}(\gamma,\gamma')\,\varepsilon
\le
\binom{L}{2}\varepsilon
\le
\binom{H}{2}\varepsilon,
\label{eq:order-insensitivity-bound}
\end{equation}
where \(N_{\mathrm{swap}}(\gamma,\gamma')\) is the minimum number of diamond swaps needed to transform \(\gamma\) into \(\gamma'\).
\end{theorem}

\begin{theorem}[Dynamic programming on the truncated ideal lattice]
\label{thm:dp-main}
Assume \textbf{C4}. Fix \(z=(I,x)\) and \(H<\infty\). Then
\begin{equation}
\textstyle \arg\max_{\gamma\in \Gamma_H(I)} Q_z(\gamma)
\label{eq:dp-target}
\end{equation}
is nonempty and computable by dynamic programming on the depth-\(H\) truncation of the upward-oriented Hasse diagram of \(J(P)\).
\end{theorem}

\section{Experiments}
\label{sec:experiments}

\begingroup
\setlength{\textfloatsep}{6pt plus 1pt minus 2pt}
\setlength{\floatsep}{5pt plus 1pt minus 2pt}
\setlength{\intextsep}{5pt plus 1pt minus 2pt}
\setlength{\abovecaptionskip}{3pt}
\setlength{\belowcaptionskip}{-2pt}
\setlength{\tabcolsep}{4pt}
\renewcommand{\arraystretch}{0.96}

This section tests whether the structural predictions of the theory are visible in real sequential logs and whether they induce nontrivial planning consequences. Concretely, we test three implications: endpoint support is broader than two-sided order support, the reference-path decomposition recovers pooled path values exactly, and sequence-sensitive planning improves decisions precisely in local families with nontrivial order signal. The code is available at: \url{https://github.com/dmivilensky/Order-Sensitive-Sequential-Interventions}.

\paragraph{Dataset, local families, and reward.}
We use the public \texttt{reviewing} event log, which contains repeated review activities and therefore naturally induces local diamonds of the form
\[
(u,w)\to v,
\]
where $u$ and $w$ are two incomparable predecessor activities that are both observed before the target activity $v$ and occur in both orders in the log. Episodes are anchored at the first occurrence of the target activity. For each family we define the four endpoint classes
$
\varnothing, \{u\}, \{w\}, \{u,w\},
$
with the two-sided order split inside the pair endpoint. The downstream reward is
\[
R
=
\mathbf 1\{\texttt{case eventually accepted}\}
-\lambda\,\texttt{remaining\_days\_after\_target},
\qquad
\lambda=0.02,
\]
and we report a sensitivity analysis over multiple values of $\lambda$ below. We retain six families built from the review activities $r_1=\texttt{get review 1}$, $r_2=\texttt{get review 2}$, and $r_3=\texttt{get review 3}$, paired with the targets \texttt{decide} and \texttt{collect reviews}.

\paragraph{Evaluation objects and baselines.}
For each family we estimate pooled endpoint values, pooled local order effects, and pooled path values. On held-out data we compare the following policies: {sequence-sensitive}, which maximizes the full value model $\Phi^{\rho}+\kappa$; {reference-path}, which ignores order corrections; {greedy one-step}; two deterministic baselines {fixed forward} and {fixed reverse}; {endpoint-pooled}; and {frequency}. We also verify that dynamic programming on the truncated ideal lattice returns the same maximizing path as exhaustive search.

\begin{table*}[t]
\centering
\footnotesize
\begin{minipage}[t]{0.58\textwidth}
\centering
\resizebox{\linewidth}{!}{\begin{tabular}{cllccc}
\toprule
ID & target & pair & $(n_{\varnothing},n_u,n_w,n_{uw})$ & $(n_{u\to w}/n_{w\to u})$ & $\hat\kappa\;[95\%\ \mathrm{CI}]$ \\
\midrule
F1 & \texttt{decide} & $(r_1,r_2)$ & $(34,21,21,24)$ & $14/10$ & $4.74\,[-0.51,\,11.29]$ \\
F2 & \texttt{decide} & $(r_1,r_3)$ & $(27,23,28,22)$ & $10/12$ & $1.68\,[-4.73,\,8.33]$ \\
F3 & \texttt{decide} & $(r_2,r_3)$ & $(28,22,27,23)$ & $17/6$ & $-6.95\,[-11.89,\,-1.42]$ \\
F4 & \texttt{collect} & $(r_1,r_2)$ & $(35,21,20,24)$ & $14/10$ & $6.13\,[0.59,\,13.70]$ \\
F5 & \texttt{collect} & $(r_1,r_3)$ & $(27,23,28,22)$ & $10/12$ & $2.84\,[-4.83,\,9.85]$ \\
F6 & \texttt{collect} & $(r_2,r_3)$ & $(29,21,27,23)$ & $17/6$ & $-7.75\,[-12.46,\,-1.75]$ \\
\bottomrule
\end{tabular}
}
\caption{Selected families, endpoint support, two-sided order support, and pooled local order effect. Here $n_{\varnothing},n_u,n_w,n_{uw}$ denote endpoint counts and $n_{u\to w}/n_{w\to u}$ denotes two-sided order support inside the pair endpoint. Families F1--F6 are used throughout the empirical section.}
\label{tab:family_summary_main}
\end{minipage}\hfill
\begin{minipage}[t]{0.40\textwidth}
\centering
\resizebox{\linewidth}{!}{\begin{tabular}{clccc}
\toprule
ID & DP argmax & exhaustive argmax & equal & best value \\
\midrule
F1 & $w\!\to\!u$ & $w\!\to\!u$ & $\checkmark$ & $-7.66$ \\
F2 & $w\!\to\!u$ & $w\!\to\!u$ & $\checkmark$ & $-7.22$ \\
F3 & $u\!\to\!w$ & $u\!\to\!w$ & $\checkmark$ & $-10.89$ \\
F4 & $w\!\to\!u$ & $w\!\to\!u$ & $\checkmark$ & $-9.30$ \\
F5 & $w\!\to\!u$ & $w\!\to\!u$ & $\checkmark$ & $-9.03$ \\
F6 & $u\!\to\!w$ & $u\!\to\!w$ & $\checkmark$ & $-12.59$ \\
\bottomrule
\end{tabular}
}
\caption{Dynamic-programming and exhaustive maximizers coincide in all six families.}
\label{tab:dp_exhaustive_main}
\end{minipage}
\end{table*}

\paragraph{Support separation and family heterogeneity.}
Table~\ref{tab:family_summary_main} and Figure~\ref{fig:empirical_top_compact} show the empirical footprint of the support-separation theory. Endpoint support is consistently nondegenerate, whereas order support is narrower and heterogeneous. The strongest positive order signal is obtained for the pair $(r_1,r_2)$: $\hat\kappa=4.74 \text{ for \texttt{decide}}$, $\hat\kappa=6.13 \text{ for \texttt{collect reviews}}$,
with two-sided support $14/10$ under both targets. By contrast, the pair $(r_2,r_3)$ induces a strong negative order effect under both targets, $\hat\kappa=-6.95 \text{ for \texttt{decide}}$,
$\hat\kappa=-7.75 \text{ for \texttt{collect reviews}}$
while $(r_1,r_3)$ is materially weaker. The event log therefore contains both high-signal and low-signal local diamonds.

\begin{figure*}[t]
    \centering
    \includegraphics[width=0.90\textwidth]{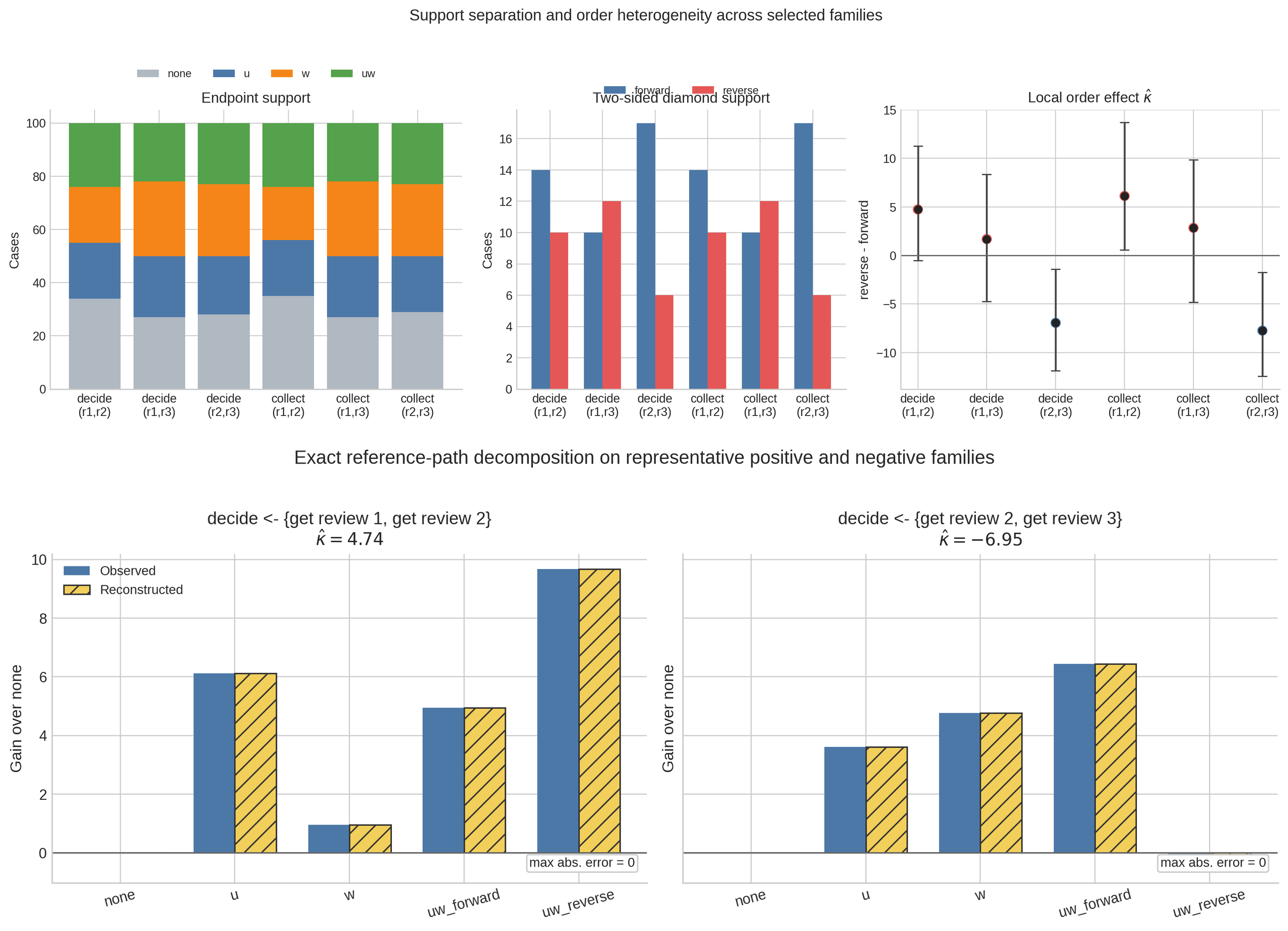}
    \caption{Top: support and order heterogeneity across the six selected families. Bottom: observed versus reconstructed pooled path values for representative positive- and negative-order families. Endpoint support is consistently broader than order support, and pooled path means are recovered exactly by the decomposition \(V(\gamma)=\Phi^{\rho}+\kappa\).}
    \label{fig:empirical_top_compact}
\end{figure*}

\paragraph{Exact decomposition of pooled path values.}
Figure~\ref{fig:empirical_top_compact} also shows the value-theoretic content of the model. For both a positive-order family and a negative-order family, the observed pooled path values are matched exactly by the reference-path decomposition. The reconstruction error is zero up to numerical precision. Empirically, the order-sensitive term is therefore neither decorative nor residual: it is exactly the quantity needed to move from the reference path score to the alternative order.

\begin{table*}[t]
\centering
\footnotesize
\resizebox{0.6\textwidth}{!}{\begin{tabular}{clccccc}
\toprule
ID & selected path & held-out value & $\Delta_{\mathrm{ref}}$ & $\Delta_{\mathrm{greedy}}$ & win rate \\
\midrule
F1 & $w\!\to\!u$ & $-7.84$ & $4.61$ & $3.62$ & $0.80$ \\
F2 & $w\!\to\!u$ & $-9.57$ & $-0.52$ & $3.99$ & $0.27$ \\
F3 & $u\!\to\!w$ & $-13.38$ & $0.00$ & $0.83$ & $0.00$ \\
F4 & $w\!\to\!u$ & $-9.22$ & $5.37$ & $3.94$ & $0.87$ \\
F5 & $w\!\to\!u$ & $-10.78$ & $1.44$ & $5.53$ & $0.57$ \\
F6 & $u\!\to\!w$ & $-15.32$ & $0.00$ & $0.38$ & $0.00$ \\
\bottomrule
\end{tabular}
}
\caption{Held-out policy summary for the six selected families. The selected path is the sequence-sensitive maximizer, $\Delta_{\mathrm{ref}}$ is its gain over the reference-path baseline, and $\Delta_{\mathrm{greedy}}$ is its gain over the greedy one-step baseline.}
\label{tab:policy_summary_main}
\end{table*}

\paragraph{Planning consequences.}
Table~\ref{tab:policy_summary_main} and Figure~\ref{fig:empirical_bottom} summarize the held-out policy comparison. Three conclusions are immediate. First, sequence-sensitive planning is selectively useful: it improves on the reference-path policy precisely in the families with strong order signal. For the positive-order pair $(r_1,r_2)$, the sequence-sensitive policy improves held-out value over the reference-path policy by approximately $4.61$ for \texttt{decide} and $5.37$ for \texttt{collect reviews}. Second, the gains are not driven by comparison to a single weak baseline: the sequence-sensitive policy is compared against reference-path, greedy one-step, both fixed orders, endpoint-pooled, and frequency baselines, and the strongest families are exactly those in which the full policy departs from simpler policies and achieves the best held-out value. Third, Table~\ref{tab:dp_exhaustive_main} shows that dynamic programming and exhaustive search agree in all six families, matching the exact planning claim of Theorem~\ref{thm:dp-main}.

\begin{figure*}[t]
    \centering
    \includegraphics[width=0.70\textwidth]{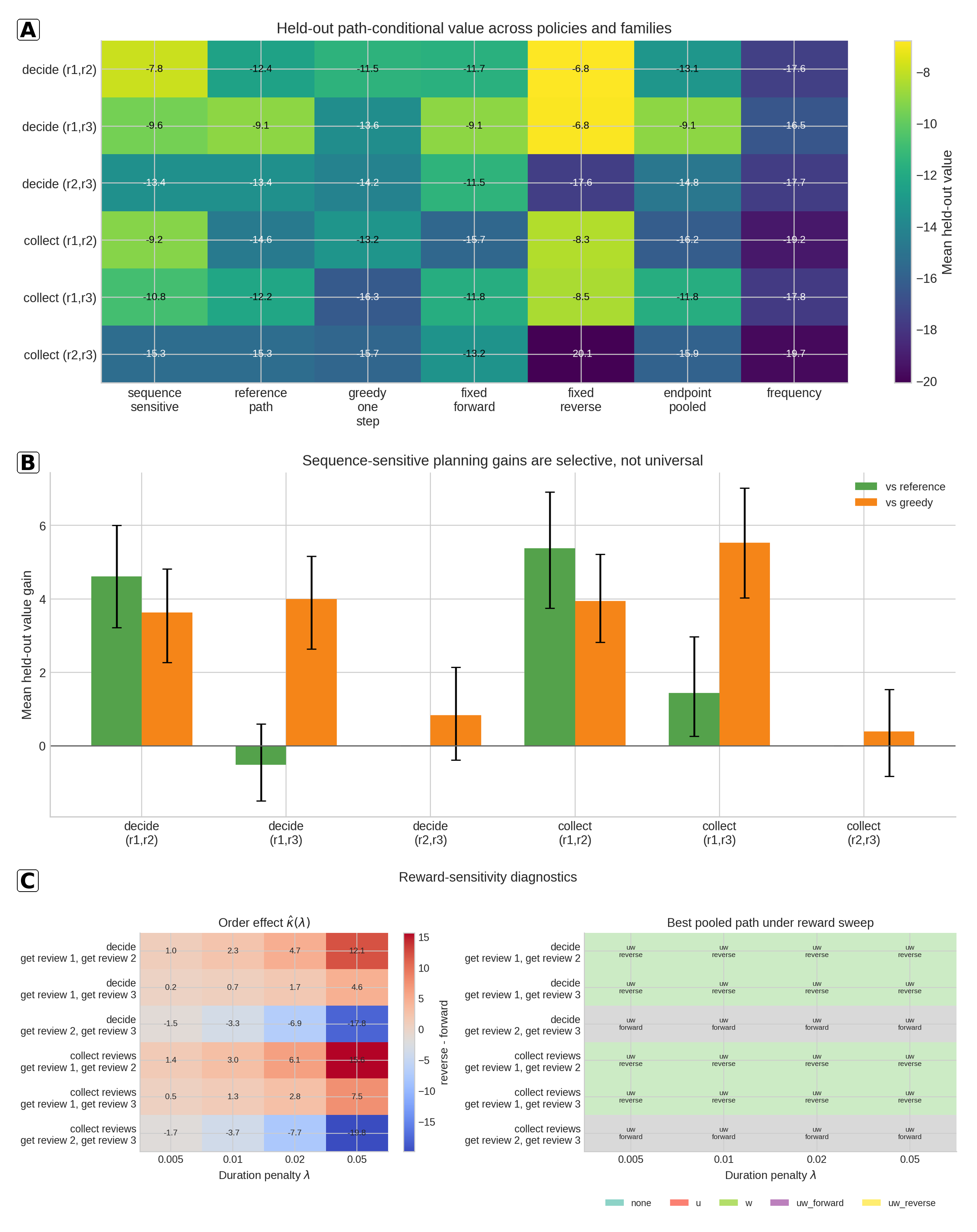}
    \caption{
Policy-level consequences of local order sensitivity.
Panel A reports held-out path-conditional value across families and baselines.
Panel B reports gains of the sequence-sensitive policy relative to the main simpler baselines.
Panel C reports reward-sensitivity diagnostics across duration penalties.
The empirical pattern is selective: sequence-sensitive planning materially improves held-out value in the families with strong local order signal and collapses toward simpler policies when the order effect is weak.
}
    \label{fig:empirical_bottom}
\end{figure*}

\endgroup

\section{Conclusion}

We developed an exact theory of order-sensitive sequential interventions under prerequisite constraints. On the ideal lattice of a finite poset, local diamonds generate all same-endpoint order variation, cube consistency characterizes realizable local order fields, and reduced-state support determines which reference-path and order-sensitive quantities are identifiable. These results also yield exact planning consequences, including an order-insensitivity bound and dynamic programming on the truncated lattice.

\newpage
\appendix

\section{Combinatorial Preliminaries}
\label{app:comb}

\begin{lemma}[Boolean interval lemma]
\label{lem:boolean-interval-app}
Let \(I\in J(P)\), and let \(A\subseteq \mathcal A(I)\) be an antichain. Then, for every \(B\subseteq A\), the set \(I\cup B\) is an ideal. Consequently,
\[
[I,I\cup A]
=
\{I\cup B:\ B\subseteq A\},
\]
and the interval \([I,I\cup A]\) is canonically isomorphic to the Boolean lattice \(B_{|A|}\).
\end{lemma}

\begin{proof}
Fix \(B\subseteq A\). Let \(s\in I\cup B\) and \(u\preceq s\). If \(s\in I\), then \(u\in I\subseteq I\cup B\) because \(I\) is an ideal. If \(s\in B\subseteq A\), then \(s\in \mathcal A(I)\), hence \(\mathrm{Pred}(s)\subseteq I\subseteq I\cup B\). Thus \(u\in I\cup B\) in all cases. Hence \(I\cup B\) is an ideal.

Now let \(K\) be any ideal with \(I\subseteq K\subseteq I\cup A\). Then \(K\setminus I\subseteq A\), so \(K=I\cup B\) for \(B=K\setminus I\subseteq A\). Therefore
\[
[I,I\cup A]=\{I\cup B:\ B\subseteq A\}.
\]
The map \(B\mapsto I\cup B\) is inclusion-preserving and bijective, hence an isomorphism with \(B_{|A|}\).
\end{proof}

\begin{lemma}[Reference paths are admissible]
\label{lem:ref-path-admissible}
For every interval \(I\subseteq J\) in \(J(P)\), the reference path \(\rho_{I,J}\) is admissible.
\end{lemma}

\begin{proof}
Write
\[
J\setminus I=\{a_1,\dots,a_m\},
\qquad
a_1<_{\tau}\cdots<_{\tau} a_m.
\]
Define
\[
I_0:=I,\qquad I_r:=I\cup\{a_1,\dots,a_r\},
\qquad r=1,\dots,m.
\]
Fix \(r<m\). Let \(u\prec a_{r+1}\). Since \(J\) is an ideal and \(a_{r+1}\in J\), one has \(u\in J\). If \(u\in I\), then \(u\in I_r\). If \(u\in J\setminus I\), then \(\tau\) extends \(\preceq\), so \(u<_{\tau}a_{r+1}\), hence \(u=a_j\) for some \(j\le r\), and therefore \(u\in I_r\). Thus \(\mathrm{Pred}(a_{r+1})\subseteq I_r\), proving \(a_{r+1}\in \mathcal A(I_r)\).
\end{proof}

\begin{lemma}[Paths are linear extensions]
\label{lem:linext-app}
Let \(I\subseteq J\) in \(J(P)\). Then \(\Gamma(I,J)\) is in canonical bijection with the set of linear extensions of the induced poset on \(J\setminus I\).
\end{lemma}

\begin{proof}
Let
\[
\gamma=(I_0,\dots,I_L)\in \Gamma(I,J).
\]
Since each step adds one new element, necessarily \(L=|J\setminus I|\). Write
\[
I_{\ell+1}=I_\ell\cup\{a_{\ell+1}\},
\qquad \ell=0,\dots,L-1.
\]
Then \(J\setminus I=\{a_1,\dots,a_L\}\). If \(a_r\prec a_s\) in the induced poset on \(J\setminus I\), then \(a_r\) must be present before \(a_s\) can be admissibly added, so \(r<s\). Thus \((a_1,\dots,a_L)\) is a linear extension.

Conversely, let \((a_1,\dots,a_L)\) be a linear extension of the induced poset on \(J\setminus I\). Define
\[
I_0:=I,\qquad I_\ell:=I\cup\{a_1,\dots,a_\ell\}.
\]
Fix \(\ell<L\). Every predecessor of \(a_{\ell+1}\) that lies in \(J\setminus I\) appears among \(a_1,\dots,a_\ell\); every predecessor already in \(I\) is present from the start. Therefore
\[
\mathrm{Pred}(a_{\ell+1})\subseteq I_\ell,
\]
so \(a_{\ell+1}\in \mathcal A(I_\ell)\). Thus \((I_0,\dots,I_L)\in \Gamma(I,J)\). The two constructions are inverse.
\end{proof}

\begin{lemma}[Adjacent-swap connectivity]
\label{lem:adj-swap-app}
Fix \(I\subseteq J\) in \(J(P)\). Any two paths in \(\Gamma(I,J)\) are connected by a finite sequence of adjacent swaps of consecutive incomparable elements in the associated linear extensions.
\end{lemma}

\begin{proof}
By Lemma~\ref{lem:linext-app}, it suffices to prove the claim for linear extensions of the finite poset \(P|_{J\setminus I}\). Let
\[
a_1,\dots,a_L
\qquad\text{and}\qquad
b_1,\dots,b_L
\]
be two linear extensions. We prove the claim by induction on \(L\).

For \(L\le 1\), there is nothing to show. Assume \(L\ge 2\). In the target extension, let \(b_L\) be the final element. Since \(b_L\) is maximal in the induced poset, every element appearing to its right in the source extension is incomparable with \(b_L\). Thus \(b_L\) can be moved rightward, one adjacent swap at a time, until it occupies the final position. Removing \(b_L\) from both extensions yields two linear extensions of a smaller poset, and the induction hypothesis applies to the first \(L-1\) positions.
\end{proof}

\begin{proposition}[Adjacent swaps are diamonds]
\label{prop:swap-diamond-app}
Let \(\gamma\in \Gamma(I,J)\), and let its associated linear extension contain a neighboring incomparable pair \((u,v)\) with \(u<_{\tau}v\). Let \(\gamma'\) be obtained by swapping these two adjacent entries. Then \(\gamma'\in \Gamma(I,J)\), and the local replacement from \(\gamma\) to \(\gamma'\) is exactly a diamond swap.
\end{proposition}

\begin{proof}
Let \(K\) be the ideal formed by the prefix preceding \(u\) and \(v\). Since \(u\) and \(v\) are consecutive and incomparable in the extension, all predecessors of either element lie already in \(K\). Therefore
\[
u,v\in \mathcal A(K).
\]
Since \(u\parallel v\), the two-step paths
\[
K\to K\cup\{u\}\to K\cup\{u,v\},
\qquad
K\to K\cup\{v\}\to K\cup\{u,v\}
\]
form a diamond, and replacing one side by the other yields \(\gamma'\).
\end{proof}

\begin{proof}[Proof of Theorem~\ref{thm:diamond-representation-main}]
By Lemma~\ref{lem:linext-app}, \(\gamma\) and \(\gamma'\) correspond to linear extensions of the same induced poset on \(J\setminus I\). By Lemma~\ref{lem:adj-swap-app}, one extension is transformed into the other by finitely many adjacent swaps of neighboring incomparable elements. By Proposition~\ref{prop:swap-diamond-app}, each such swap is a diamond swap. Recording the orientation of each swap yields \eqref{eq:diamond-representation}.
\end{proof}

\section{Proofs for Curvature and Exact Decomposition}
\label{app:struct}

\begin{proof}[Proof of Theorem~\ref{thm:path-independence-main}]
{\((3)\Rightarrow(1)\).}
Let \(\gamma=(I_0,\dots,I_L)\in \Gamma(I,J)\) with
\[
I_{\ell+1}=I_\ell\cup\{a_\ell\}.
\]
Then
\[
V_x(\gamma)
=
\sum_{\ell=0}^{L-1}\bigl(\Phi_x(I_{\ell+1})-\Phi_x(I_\ell)\bigr)
=
\Phi_x(J)-\Phi_x(I).
\]
Hence \(V_x(\gamma)\) depends only on the endpoints.

\smallskip
{\((1)\Rightarrow(2)\).}
Apply endpoint-independence to the two boundary paths of any diamond.

\smallskip
{\((2)\Rightarrow(3)\).}
Fix \(I\in J(P)\). Define \(\Phi_x(I)\) as the value of any admissible path from \(\varnothing\) to \(I\). This is well-defined: if \(\rho,\rho'\in \Gamma(\varnothing,I)\), then by Theorem~\ref{thm:diamond-representation-main} one is obtained from the other by a finite sequence of diamond swaps. Since each swap changes path value by the corresponding curvature and all curvatures vanish, \(V_x(\rho)=V_x(\rho')\).

Now let \(a\in \mathcal A(I)\). For any \(\rho\in \Gamma(\varnothing,I)\), the concatenation \(\rho\cdot (I\to I\cup\{a\})\) is admissible from \(\varnothing\) to \(I\cup\{a\}\). Therefore
\[
\Phi_x(I\cup\{a\})
=
V_x(\rho)+g_x(I,a)
=
\Phi_x(I)+g_x(I,a),
\]
which yields \eqref{eq:endpoint-potential-gradient}. Equation \eqref{eq:path-independence-telescoping} was already established in the proof of \((3)\Rightarrow(1)\).
\end{proof}

\begin{lemma}[Swap-sum identity]
\label{lem:swap-sum-app}
Let \(\gamma,\gamma'\in \Gamma(I,J)\), and let \((d_m,\varepsilon_m)_{m=1}^M\) be a diamond-swap sequence taking \(\gamma\) to \(\gamma'\). Then
\begin{equation}
V_x(\gamma')-V_x(\gamma)=\sum_{m=1}^M \varepsilon_m\,\kappa_x(d_m).
\label{eq:swap-sum-identity}
\end{equation}
\end{lemma}

\begin{proof}
Proceed by induction on \(M\). For \(M=1\), the statement is exactly the definition of \(\kappa_x\). Assume the statement for \(M-1\), and let \(\tilde\gamma\) be the intermediate path after the first \(M-1\) swaps. Then
\[
V_x(\gamma')-V_x(\gamma)
=
\bigl(V_x(\gamma')-V_x(\tilde\gamma)\bigr)
+
\bigl(V_x(\tilde\gamma)-V_x(\gamma)\bigr).
\]
The first term is \(\varepsilon_M\kappa_x(d_M)\), and the second term is \(\sum_{m=1}^{M-1}\varepsilon_m\kappa_x(d_m)\) by the induction hypothesis.
\end{proof}

\begin{proof}[Proof of Theorem~\ref{thm:reference-decomposition-main}]
Apply Lemma~\ref{lem:swap-sum-app} with \(\gamma=\rho_{I,J}\) and \(\gamma'=\gamma\). The identity \eqref{eq:reference-path-score} gives \eqref{eq:reference-path-decomposition}.
\end{proof}

\begin{proof}[Proof of Corollary~\ref{cor:mobius-main}]
Let \(\zeta\) denote the zeta function of the finite poset \(J(P)\):
\[
\zeta(K,I)=\one\{K\subseteq I\}.
\]
By standard incidence-algebra theory \cite{rota1964,stanley2011ec1}, \(\zeta\) is invertible under convolution, with inverse \(\mu\), the M\"obius function. The representation \eqref{eq:mobius-forward} is precisely
\[
\Phi_x=\zeta * \theta_x.
\]
Hence
\[
\theta_x=\mu * \Phi_x,
\]
which yields \eqref{eq:mobius-inverse}. Uniqueness follows from invertibility of \(\zeta\).

Under the assumptions of Theorem~\ref{thm:path-independence-main},
\[
g_x(I,a)=\Phi_x(I\cup\{a\})-\Phi_x(I).
\]
Subtracting the M\"obius expansions for \(I\cup\{a\}\) and \(I\) cancels exactly the terms \(K\) not containing \(a\), and leaves \eqref{eq:mobius-edge}.
\end{proof}

\section{Proofs for Cube Consistency and Integrability}
\label{app:integrability}

\subsection{The reference tree}

\begin{lemma}[The reference tree is a rooted spanning tree]
\label{lem:tree-app}
The directed graph \(T_\tau\) with vertex set \(J(P)\) and edges
\[
K^- \to K,\qquad K\in J(P)\setminus\{\varnothing\},
\]
is a rooted spanning tree of the upward-oriented Hasse diagram of \(J(P)\), rooted at \(\varnothing\).
\end{lemma}

\begin{proof}
Every nonempty ideal \(K\) has exactly one incoming edge in \(T_\tau\), namely
\[
K^- \to K.
\]
The empty ideal has no incoming edge. Thus \(T_\tau\) is a rooted arborescence if every vertex is reachable from \(\varnothing\). But iterating
\[
K\mapsto K^-
\]
strictly decreases cardinality and eventually reaches \(\varnothing\), so every \(K\) is connected to \(\varnothing\) by a unique directed path. Hence \(T_\tau\) is a rooted spanning tree.
\end{proof}

\subsection{Necessity of cube consistency}

\begin{proof}[Proof of Proposition~\ref{prop:bianchi-main}]
Fix a three-cube \((I;u,v,w)\) with
\[
u<_{\tau}v<_{\tau}w.
\]
Expanding the six curvature terms in \eqref{eq:cube-consistency} gives
\begin{align*}
&\kappa(I;u,v)
-\kappa(I\cup\{w\};u,v)
+\kappa(I;u,w)
-\kappa(I\cup\{v\};u,w)
+\kappa(I;v,w)
-\kappa(I\cup\{u\};v,w)\\
&=
\bigl[g(I,v)+g(I\cup\{v\},u)-g(I,u)-g(I\cup\{u\},v)\bigr]\\
&\quad
-\bigl[g(I\cup\{w\},v)+g(I\cup\{v,w\},u)-g(I\cup\{w\},u)-g(I\cup\{u,w\},v)\bigr]\\
&\quad
+\bigl[g(I,w)+g(I\cup\{w\},u)-g(I,u)-g(I\cup\{u\},w)\bigr]\\
&\quad
-\bigl[g(I\cup\{v\},w)+g(I\cup\{v,w\},u)-g(I\cup\{v\},u)-g(I\cup\{u,v\},w)\bigr]\\
&\quad
+\bigl[g(I,w)+g(I\cup\{w\},v)-g(I,v)-g(I\cup\{v\},w)\bigr]\\
&\quad
-\bigl[g(I\cup\{u\},w)+g(I\cup\{u,w\},v)-g(I\cup\{u\},v)-g(I\cup\{u,v\},w)\bigr].
\end{align*}
Every edge term appears exactly twice with opposite signs, so the total is zero. This is precisely \eqref{eq:cube-consistency}.
\end{proof}

\subsection{Constructive zero-gauge integrability}

\begin{theorem}[Zero-gauge integrability]
\label{thm:zero-gauge-app}
Let \(\kappa\) be a cube-consistent diamond field. Then there exists a unique edge field \(g^{(0)}\) such that
\begin{enumerate}
    \item \(g^{(0)}(K^-,m(K))=0\) for every nonempty ideal \(K\);
    \item the curvature field of \(g^{(0)}\) equals \(\kappa\).
\end{enumerate}
\end{theorem}

\begin{proof}
For an admissible edge \((I,a)\), define
\begin{equation}
\delta(I,a):=\bigl|\{b\in I:\ a<_{\tau} b\}\bigr|.
\label{eq:delta-recursion}
\end{equation}
We construct \(g^{(0)}(I,a)\) recursively on \(\delta(I,a)\).

\paragraph{Base case.}
If \(\delta(I,a)=0\), then \(a\) is the \(<_{\tau}\)-maximum element of \(I\cup\{a\}\), so
\[
(I,a)=((I\cup\{a\})^-,m(I\cup\{a\})).
\]
Define
\begin{equation}
g^{(0)}(I,a):=0.
\label{eq:zero-gauge-base}
\end{equation}

\paragraph{Inductive step.}
Assume \(g^{(0)}\) has been defined on all edges \((K,c)\) with \(\delta(K,c)<r\), and let \((I,a)\) satisfy \(\delta(I,a)=r\ge 1\). Let
\[
b:=m(I),\qquad K:=I\setminus\{b\}.
\]
We first verify that \((K;a,b)\) is a diamond. Since \(a<_{\tau}b\), one has \(a\parallel b\): if \(a\preceq b\), then \(a\in I\) because \(I\) is an ideal, contradiction; if \(b\preceq a\), then \(b<_{\tau}a\) because \(\tau\) extends \(\preceq\), contradiction. Next, \(b\in \mathcal A(K)\) because every predecessor of \(b\) lies in \(I\), and none equals \(b\). Also \(a\in \mathcal A(K)\) because every predecessor of \(a\) lies in \(I\), and \(b\notin \mathrm{Pred}(a)\). Thus \((K;a,b)\) is a diamond.

Now define
\begin{equation}
g^{(0)}(I,a):=g^{(0)}(K,a)+\kappa(K;a,b).
\label{eq:zero-gauge-recursion}
\end{equation}
This is valid because
\[
\delta(K,a)=\delta(I,a)-1.
\]

We now prove by induction on
\[
\Delta(I;u,v):=\bigl|\{b\in I:\ v<_{\tau} b\}\bigr|
\]
that the curvature of \(g^{(0)}\) on every diamond \((I;u,v)\), \(u<_{\tau}v\), equals \(\kappa(I;u,v)\).

\paragraph{Base case.}
Assume \(\Delta(I;u,v)=0\). Then \(v\) is the \(<_{\tau}\)-maximum element of \(I\cup\{u,v\}\), so both edges
\[
I\to I\cup\{v\},
\qquad
I\cup\{u\}\to I\cup\{u,v\}
\]
are reference-tree edges, hence have \(g^{(0)}\)-value \(0\). Therefore
\[
\kappa_{g^{(0)}}(I;u,v)=g^{(0)}(I\cup\{v\},u)-g^{(0)}(I,u).
\]
Applying the recursion \eqref{eq:zero-gauge-recursion} to \((I\cup\{v\},u)\) with \(K=I\) and \(b=v\) gives
\[
g^{(0)}(I\cup\{v\},u)=g^{(0)}(I,u)+\kappa(I;u,v).
\]
Hence \(\kappa_{g^{(0)}}(I;u,v)=\kappa(I;u,v)\).

\paragraph{Inductive step.}
Assume the claim whenever \(\Delta<r\), and let \(\Delta(I;u,v)=r\ge 1\). Set
\[
b:=m(I),\qquad K:=I\setminus\{b\}.
\]
Then \(u,v,b\) are pairwise incomparable. Applying \eqref{eq:zero-gauge-recursion} to the four non-tree edges in the curvature formula gives
\begin{align*}
g^{(0)}(I,v)&=g^{(0)}(K,v)+\kappa(K;v,b),\\
g^{(0)}(I\cup\{v\},u)&=g^{(0)}(K\cup\{v\},u)+\kappa(K\cup\{v\};u,b),\\
g^{(0)}(I,u)&=g^{(0)}(K,u)+\kappa(K;u,b),\\
g^{(0)}(I\cup\{u\},v)&=g^{(0)}(K\cup\{u\},v)+\kappa(K\cup\{u\};v,b).
\end{align*}
Hence
\begin{align*}
\kappa_{g^{(0)}}(I;u,v)
&=
\kappa_{g^{(0)}}(K;u,v)
+\kappa(K;v,b)+\kappa(K\cup\{v\};u,b)-\kappa(K;u,b)-\kappa(K\cup\{u\};v,b).
\end{align*}
By the induction hypothesis,
\[
\kappa_{g^{(0)}}(K;u,v)=\kappa(K;u,v).
\]
Cube consistency on the three-cube \((K;u,v,b)\) gives
\[
\kappa(I;u,v)
=
\kappa(K;u,v)
+\kappa(K;v,b)+\kappa(K\cup\{v\};u,b)-\kappa(K;u,b)-\kappa(K\cup\{u\};v,b).
\]
Comparing the two displays yields
\[
\kappa_{g^{(0)}}(I;u,v)=\kappa(I;u,v).
\]

This proves existence. For uniqueness, let \(\tilde g^{(0)}\) be another such field. Then
\[
h:=g^{(0)}-\tilde g^{(0)}
\]
has zero curvature, so by Theorem~\ref{thm:path-independence-main} there exists \(\varphi:J(P)\to\RR\) such that
\[
h(I,a)=\varphi(I\cup\{a\})-\varphi(I).
\]
Since \(h\) vanishes on every reference-tree edge,
\[
0=h(K^-,m(K))=\varphi(K)-\varphi(K^-)
\qquad \forall K\neq \varnothing.
\]
By induction along the reference tree, \(\varphi\) is constant on all ideals, hence \(h=0\). Therefore \(g^{(0)}=\tilde g^{(0)}\).
\end{proof}

\subsection{General gauge fixing}

\begin{lemma}[Tree integration]
\label{lem:tree-potential-app}
Let \(\alpha:J(P)\setminus\{\varnothing\}\to\RR\). Then there exists a unique function
\[
\psi:J(P)\to\RR
\]
such that
\begin{equation}
\psi(\varnothing)=0,
\qquad
\psi(K)-\psi(K^-)=\alpha(K)
\qquad
\forall K\neq \varnothing.
\label{eq:tree-integration}
\end{equation}
\end{lemma}

\begin{proof}
Existence and uniqueness follow by recursion along the rooted spanning tree \(T_\tau\) from Lemma~\ref{lem:tree-app}. Explicitly,
\[
\psi(K)=\sum_{L\in \rho_{\varnothing,K}} \alpha(L),
\]
where the sum is taken over the nonempty ideals visited by the unique path from \(\varnothing\) to \(K\) in \(T_\tau\).
\end{proof}

\begin{lemma}[Gradient shifts preserve curvature]
\label{lem:gradient-app}
Let \(g\) be an edge field and let \(\psi:J(P)\to\RR\). Define
\begin{equation}
(g+d\psi)(I,a):=g(I,a)+\psi(I\cup\{a\})-\psi(I).
\label{eq:gradient-shift}
\end{equation}
Then \(g+d\psi\) and \(g\) have the same curvature field.
\end{lemma}

\begin{proof}
For a diamond \((I;u,v)\), \(u<_{\tau}v\),
\begin{align*}
\kappa_{g+d\psi}(I;u,v)
&=
(g+d\psi)(I,v)+(g+d\psi)(I\cup\{v\},u)\\
&\quad -(g+d\psi)(I,u)-(g+d\psi)(I\cup\{u\},v)\\
&=
\kappa_g(I;u,v)\\
&\quad
+\bigl[\psi(I\cup\{v\})-\psi(I)\bigr]
+\bigl[\psi(I\cup\{u,v\})-\psi(I\cup\{v\})\bigr]\\
&\quad
-\bigl[\psi(I\cup\{u\})-\psi(I)\bigr]
-\bigl[\psi(I\cup\{u,v\})-\psi(I\cup\{u\})\bigr]\\
&=
\kappa_g(I;u,v).
\end{align*}
\end{proof}

\begin{proof}[Proof of Theorem~\ref{thm:integrability-main}]
Let \(g^{(0)}\) be the zero-gauge field from Theorem~\ref{thm:zero-gauge-app}. By Lemma~\ref{lem:tree-potential-app}, there is a unique \(\psi\) with \(\psi(\varnothing)=0\) and \eqref{eq:tree-integration}. Define
\[
g:=g^{(0)}+d\psi.
\]
By Lemma~\ref{lem:gradient-app}, \(g\) has the same curvature field \(\kappa\) as \(g^{(0)}\). For every nonempty ideal \(K\),
\[
g(K^-,m(K))
=
g^{(0)}(K^-,m(K))+\psi(K)-\psi(K^-)
=
0+\alpha(K)
=
\alpha(K).
\]
Thus \(g\) satisfies both requirements.

For uniqueness, let \(\tilde g\) be another solution. Then
\[
h:=g-\tilde g
\]
has zero curvature and vanishes on the reference tree. By the uniqueness part of Theorem~\ref{thm:zero-gauge-app}, \(h=0\). Hence \(g=\tilde g\).
\end{proof}

\section{Proofs for Causal Identification}
\label{app:causal}

\begin{proof}[Proof of Proposition~\ref{prop:nonid-main}]
Let
\[
\gamma_u:
I\to I\cup\{u\}\to I\cup\{u,v\},
\qquad
\gamma_v:
I\to I\cup\{v\}\to I\cup\{u,v\}.
\]
By assumption,
\[
\Pr(\gamma_v\mid Z_t=(I,x))=0
\qquad \forall x\in B.
\]
Fix any model in the class satisfying \textbf{C1}--\textbf{C3}, and let \(Q^{(1)}_z(\cdot)\) denote one version of its statewise conditional counterfactual mean system.

Choose any bounded measurable function \(\Delta:B\to\RR\) that is not almost everywhere zero. Define a second counterfactual mean system \(Q^{(2)}\) by
\begin{equation}
Q^{(2)}_z(\gamma)=
\begin{cases}
Q^{(1)}_{(I,x)}(\gamma_v)+\Delta(x), & z=(I,x),\ x\in B,\ \gamma=\gamma_v,\\
Q^{(1)}_z(\gamma), & \text{otherwise}.
\end{cases}
\label{eq:nonid-perturbation}
\end{equation}
Keep the observational action law and the observational transition law unchanged.

Because \(\gamma_v\) has zero observational probability on \(B\), replacing its conditional mean by \eqref{eq:nonid-perturbation} does not alter any observed conditional expectation, any observed action probability, or any observed transition law. Hence the two models induce the same observational law. However,
\[
\kappa^{(2)}_{(I,x)}(I;u,v)
=
Q^{(2)}_{(I,x)}(\gamma_v)-Q^{(2)}_{(I,x)}(\gamma_u)
=
\kappa^{(1)}_{(I,x)}(I;u,v)+\Delta(x)
\]
for \(x\in B\). Therefore \(\kappa_{(I,x)}(I;u,v)\) is not point-identified on \(B\).
\end{proof}

\begin{proof}[Proof of Theorem~\ref{thm:gformula-main}]
Let
\[
\gamma=(a_0,\dots,a_{L-1}).
\]
For \(0\le b\le L\), write
\[
\gamma_{b:L-1}:=(a_b,\dots,a_{L-1})
\]
for the suffix of \(\gamma\) starting at stage \(b\). We prove by backward induction on \(b\) that
\begin{align}
Q_{z_b}(\gamma_{b:L-1})
=
&\int
\EE\!\left[
R
\mid
Z_L=z_L,\ A_b=a_b,\dots,A_{L-1}=a_{L-1},\ Z_b=z_b
\right] \times\\
&\times
\prod_{r=b}^{L-1} p(z_{r+1}\mid z_r,a_r)\,dz_{b+1:L}
\label{eq:gformula-backward}
\end{align}
for every \(z_b\) on the support of the intervention-induced state law.

For \(b=L\), the path suffix is empty and \eqref{eq:gformula-backward} reduces to the tautology
\[
Q_{z_L}(\varnothing)=\EE[R\mid Z_L=z_L,\ A_0=a_0,\dots,A_{L-1}=a_{L-1},\ Z_0=z_0].
\]

Assume \eqref{eq:gformula-backward} holds for \(b+1\). Then, using consistency and iterated expectation,
\begin{align*}
Q_{z_b}(\gamma_{b:L-1})
&=
\EE\!\left[
R^{\gamma_{b:L-1}}
\mid Z_b=z_b
\right]\\
&=
\int
\EE\!\left[
R^{\gamma_{b+1:L-1}}
\mid Z_{b+1}=z_{b+1},\ A_b=a_b,\ Z_b=z_b
\right]
p(z_{b+1}\mid z_b,a_b)\,dz_{b+1}.
\end{align*}
By reduced-state sequential ignorability,
\[
\EE\!\left[
R^{\gamma_{b+1:L-1}}
\mid Z_{b+1}=z_{b+1},\ A_b=a_b,\ Z_b=z_b
\right]
=
Q_{z_{b+1}}(\gamma_{b+1:L-1}).
\]
Applying the induction hypothesis to the right-hand side and substituting proves \eqref{eq:gformula-backward} for stage \(b\). Taking \(b=0\) yields \eqref{eq:gformula-supported-path}. Positivity ensures that all conditional laws are evaluated on the support of the intervention-induced state law.
\end{proof}

\begin{proof}[Proof of Theorem~\ref{thm:support-separation-main}]
Parts (1) and (2) follow by applying Theorem~\ref{thm:gformula-main} to the supported reference path and to the two supported sides of the diamond, respectively. For part (3), assume in addition \textbf{C4}. Then \(Q_z\) is edge-additive on admissible paths from the initial ideal, so Theorem~\ref{thm:reference-decomposition-main} applies pointwise to \(Q_z\). Substituting the identified quantities from parts (1)--(2) gives \eqref{eq:identified-full-decomposition}.
\end{proof}

\begin{proof}[Proof of Corollary~\ref{cor:policy-main}]
Because \(H<\infty\) and \(P\) is finite, \(\Gamma_H(I)\) is finite for every \(I\). By assumption, each candidate path has an identified value by Theorem~\ref{thm:support-separation-main}. Therefore the finite argmax set is identified. Deterministic tie-breaking makes the policy single-valued.
\end{proof}

\section{Proofs for Planning Consequences}
\label{app:planning}

\begin{lemma}[Inversion bound]
\label{lem:inversion-app}
Let \(\pi,\pi'\) be two linear extensions of the same \(L\)-element poset. Then the number of adjacent swaps needed to transform \(\pi\) into \(\pi'\) is at most \(\binom{L}{2}\).
\end{lemma}

\begin{proof}
View \(\pi\) and \(\pi'\) as two permutations of the same labeled set. The inversion count of \(\pi\) relative to \(\pi'\) is at most \(\binom{L}{2}\). Each adjacent swap changes the inversion count by exactly \(1\), so at most \(\binom{L}{2}\) swaps are needed.
\end{proof}

\begin{proof}[Proof of Theorem~\ref{thm:order-bound-main}]
By Theorem~\ref{thm:diamond-representation-main}, there exists a diamond-swap sequence
\[
(d_1,\varepsilon_1),\dots,(d_M,\varepsilon_M)
\]
transforming \(\gamma\) into \(\gamma'\), with
\[
M=N_{\mathrm{swap}}(\gamma,\gamma').
\]
Under \textbf{C4}, the structural decomposition applies to \(Q_z\), so
\[
Q_z(\gamma')-Q_z(\gamma)=\sum_{m=1}^M \varepsilon_m\,\kappa_z(d_m).
\]
Therefore
\[
|Q_z(\gamma')-Q_z(\gamma)|
\le
\sum_{m=1}^M |\kappa_z(d_m)|
\le
M\varepsilon.
\]
This proves the first inequality in \eqref{eq:order-insensitivity-bound}. The second follows from Lemma~\ref{lem:inversion-app}, and the third from \(L\le H\).
\end{proof}

\begin{proof}[Proof of Theorem~\ref{thm:dp-main}]
The upward-oriented Hasse diagram of \(J(P)\) is acyclic because each edge increases cardinality by \(1\). Restrict to the finite induced subgraph consisting of all ideals reachable from \(I\) within at most \(H\) admissible additions. On this DAG, maximizing the path value is a longest-path problem.

Under \textbf{C4}, each candidate value satisfies
\[
Q_z(\gamma)=\sum_{e\in\gamma} g_z(e).
\]
Define recursively
\[
U_H(K):=0
\]
for every reachable ideal \(K\), and for \(h=H-1,H-2,\dots,0\),
\begin{equation}
U_h(K):=
\max\Bigl\{
0,\ 
\max_{a\in \mathcal A(K)}
\bigl[g_z(K,a)+U_{h+1}(K\cup\{a\})\bigr]
\Bigr\}.
\label{eq:dp-recursion}
\end{equation}
Then \(U_0(I)\) equals the optimal value over \(\Gamma_H(I)\), and any maximizing action at each recursion step yields an optimal path. Since the recursion is over a finite DAG, the argmax is nonempty and computable.
\end{proof}



\newpage

\begin{thebibliography}{99}\setlength{\itemsep}{0.2em}

\bibitem{birkhoff1937}
Garrett Birkhoff.
\newblock Rings of sets.
\newblock {Duke Mathematical Journal}, 3(3):443--454, 1937.

\bibitem{davey2002}
Brian A. Davey and Hilary A. Priestley.
\newblock {Introduction to Lattices and Order}.
\newblock Cambridge University Press, second edition, 2002.

\bibitem{doignon1999}
Jean-Paul Doignon and Jean-Claude Falmagne.
\newblock {Knowledge Spaces}.
\newblock Springer, 1999.

\bibitem{doignon2011}
Jean-Paul Doignon and Jean-Claude Falmagne.
\newblock {Learning Spaces}.
\newblock Springer, 2011.

\bibitem{hernanrobins2020}
Miguel A. Hern\'an and James M. Robins.
\newblock {Causal Inference: What If}.
\newblock Chapman \& Hall/CRC, 2020.

\bibitem{murphy2003}
Susan A. Murphy.
\newblock Optimal dynamic treatment regimes.
\newblock {Journal of the Royal Statistical Society: Series B}, 65(2):331--355, 2003.

\bibitem{murphy2005}
Susan A. Murphy.
\newblock An experimental design for the development of adaptive treatment strategies.
\newblock {Statistics in Medicine}, 24(10):1455--1481, 2005.

\bibitem{pruesse1994}
Gara Pruesse and Frank Ruskey.
\newblock Generating linear extensions fast.
\newblock {SIAM Journal on Computing}, 23(2):373--386, 1994.

\bibitem{qian2011}
Min Qian and Susan A. Murphy.
\newblock Performance guarantees for individualized treatment rules.
\newblock {The Annals of Statistics}, 39(2):1180--1210, 2011.

\bibitem{robins1986}
James M. Robins.
\newblock A new approach to causal inference in mortality studies with a sustained exposure period---application to control of the healthy worker survivor effect.
\newblock {Mathematical Modelling}, 7(9--12):1393--1512, 1986.

\bibitem{robins1987}
James M. Robins.
\newblock Addendum to ``A new approach to causal inference in mortality studies with a sustained exposure period---application to control of the healthy worker survivor effect''.
\newblock {Computers \& Mathematics with Applications}, 14(9--12):923--945, 1987.

\bibitem{rota1964}
Gian-Carlo Rota.
\newblock On the foundations of combinatorial theory I. Theory of M\"obius functions.
\newblock {Zeitschrift f\"ur Wahrscheinlichkeitstheorie und Verwandte Gebiete}, 2:340--368, 1964.

\bibitem{stanley2011ec1}
Richard P. Stanley.
\newblock {Enumerative Combinatorics, Volume 1}.
\newblock Cambridge University Press, second edition, 2011.

\bibitem{trotter1992}
William T. Trotter.
\newblock {Combinatorics and Partially Ordered Sets: Dimension Theory}.
\newblock Johns Hopkins University Press, 1992.

\end{thebibliography}
\end{document}